\documentclass[11pt]{article}
\usepackage{amssymb}
\usepackage{amsmath}
\usepackage{mathrsfs}
\usepackage{graphics}
\usepackage{graphicx}
\usepackage{xcolor}
\usepackage{subfigure}
\usepackage[T1]{fontenc}
\usepackage{latexsym,amssymb,amsmath,amsfonts,amsthm}\usepackage{txfonts}
\newcommand{\be}{\begin{eqnarray}}
\newcommand{\ben}{\begin{eqnarray*}}
\newcommand{\en}{\end{eqnarray}}
\newcommand{\enn}{\end{eqnarray*}}

\newtheorem{theorem}{Theorem}[section]
\newtheorem{lemma}{Lemma}[section]
\newtheorem{prp}[theorem]{Proposition}
\newtheorem{thm}[theorem]{Theorem}

\newtheorem{dfn}{Definition}[section]

\newtheorem{remark}{Remark}

\begin{document}
\renewcommand{\theequation}{\arabic{section}.\arabic{equation}}
\begin{titlepage}
\title{\bf Ergodicity for a class of semilinear stochastic partial differential equations
}
\author{\ \ Zhao Dong$^{1,2}$,\ \ Rangrang Zhang$^{3,}$\thanks{Corresponding author.}\\
{\small $^1$ RCSDS, Academy of Mathematics and Systems Science, Chinese Academy of Sciences, Beijing 100190, China.}\\
{\small $^2$  School of Mathematical Sciences, University of Chinese Academy of Sciences.}\\
{\small $^3$ Department of  Mathematics,
Beijing Institute of Technology, Beijing, 100081, China}\\
({\sf dzhao@amt.ac.cn}, {\sf rrzhang@amss.ac.cn})}
\date{}
\end{titlepage}
\maketitle

\noindent\textbf{Abstract}: In this paper, we establish the existence and uniqueness of invariant measures for a class of semilinear stochastic partial differential equations driven by multiplicative noise on a bounded domain. The main results can be applied to SPDEs of various types such as the stochastic Burgers equation and the reaction-diffusion equations perturbed by space-time white noise.



\section{Introduction}
In this paper, we are concerned the following semilinear stochastic partial differential equations (SPDE):
\begin{eqnarray}\label{e-1}
\frac{\partial u(t,x)}{\partial t}=\frac{\partial^2 u(t,x)}{\partial x^2}+b(t,x,u(t,x))+\frac{\partial g(t,x, u(t,x))}{\partial x}+\sigma(t,x,u(t,x))\frac{\partial^2}{\partial t\partial x}W(t,x)
\end{eqnarray}
with Dirichlet boundary condition
\begin{eqnarray*}
u(t,0)=u(t,1)=0, \quad t\in [0,T]
\end{eqnarray*}
and the initial condition
\begin{eqnarray*}
u(0,x)=f(x)\in L^2([0,1]),
\end{eqnarray*}
where $W(t,x)$ denotes the Brownian sheet on a filterd probability space $(\Omega, \mathcal{F}, \{\mathcal{F}_t\}, P)$ with expectation $E$. The functions $b=b(t,x,r)$, $g=g(t,x,r)$, $\sigma=\sigma(t,x,r)$ are Borel functions of $(t,x,r)\in \mathbb{R}^+\times [0,1]\times \mathbb{R}$. Linear growth on $b$ and quadratic growth on $g$ are assumed in subsection \ref{sec-1}. Hence, the semilinear SPDE (\ref{e-1}) contains both the stochastic Burgers equation and the stochastic reaction-diffusion equations as special cases.

There are several recent works about the semilinear SPDE (\ref{e-1}). We only mention
two of them which are relevant to our work. The existence and uniqueness of solutions to (\ref{e-1}) was studied by Gy\"{o}ngy in \cite{G98}, where the author established global well-posedness of (\ref{e-1}) in the space $C([0,T]; L^2([0,1]))$. Based on \cite{G98}, Foondun and Setayeshgar \cite{FS17} proved the large deviations principle uniformly on compact subsets of $C([0,T]; L^2([0,1]))$ for the law of the solutions to (\ref{e-1}).

The present paper is devoted to the ergodicity of the semilinear SPDE (\ref{e-1}). Firstly, we prove the existence of invariant measures by utilizing the Krylov-Bogolyubov theorem (for more details on this theorem, see \cite{D-Z}). During the proof process, the tightness of solutions to (\ref{e-1}) in $C([0,T]; L^2([0,1]))$ plays a key role.
Secondly, we establish the uniqueness of invariant measures of (\ref{e-1}). To achieve it, we apply the Doob's method (see \cite{D-Z}). Based on this method, our proof is twofold. For the strong Feller property, we apply the strategy of truncation. It's worth mentioning that it is a quite effective technique for handling locally Lipschitz nonlinearities in stochastic equations. To learn more about this method, we refer the readers to \cite{D-G,D-P,D-Z-1,F-M} and so on. Utilizing the Bismut-Elworthy-Li formula, the strong Feller property of the truncating equations is obtained. Further, with the aid of weak-strong uniqueness principle in \cite{F-M}, we deduce that the semigroup associated with (\ref{e-1}) is strong Feller. For the irreducibility, it can be transformed to a control problem. The truncating equations is also crucial to our proof. By making energy estimates and using Girsanov theorem, we firstly obtain the irreducibility of the truncating equations. Then, due to the fact that the solution of truncating equations converges to the solution of (\ref{e-1}) in probability, we finally conclude the irreducibility of (\ref{e-1}).

\

This paper is organized as follows. The mathematical formulation of the semilinear stochastic partial differential equations and main results are in Sect. 2. The existence of invariant measures is proved in Sect. 3. The uniqueness of invariant measures is established by proving the strong Feller property and the irreducible property of (\ref{e-1}), whose proof is in Sect. 4. Finally, application to some examples are presented in Sect. 5.

\section{Framework and statement of main result}
Let $L^p([0,1]), p\in(0,\infty]$ be the Lebesgue space, whose norm is denoted by $|\cdot|_p$.
In particular, denote $H=L^2([0,1])$ with the corresponding norm  $|\cdot|_H$ and inner product $(\cdot,\cdot)_H$.

Define an operator $A:= \frac{\partial^2}{\partial x^2}$. Let $G_{t}(x,y)=G(t,x,y), t\geq0, x,y\in [0,1]$ be the Green function for the operator $\partial_t-A$  with Dirichlet boundary condition. Then, it satisfies that
\begin{eqnarray}\label{e-29}
\partial_t G_{t}(x,y)=AG_{t}(x,y).
\end{eqnarray}
Moreover, referring to \cite{Wu}, we have the following property:
\begin{eqnarray}\label{e-28}
\int^1_0G_{t}(x,y)G_{s}(y,z)dy=G_{t+s}(x,z),\quad {\rm{for}}\ t,s\geq0,\  x,z\in [0,1].
\end{eqnarray}
Let $\{e_n(x)\}_{n\geq 1}$ be the eigenvectors of  $A$ (equipped with the Dirichlet boundary) constituting an orthonormal system of $H$.
 Put
 \[
 \beta_n(t)=\int^t_0\int^1_0 e_n(x)W(dsdx),
 \]
 then, $\{\beta_n(t), n\geq 1\}$ is a sequence of independent Brownian motions. Define
 an $H-$cylindrical Brownian motion by $W(t)=\sum^{\infty}_{n=1}\beta_n e_n$ and
a mapping $\Sigma(\cdot)$ by
\[
\Sigma(f)h(x)=\sigma(f(x))h(x), \ f,h\in H.
\]
Then, the following
\[
\sigma(u(t,x,f))\dot{W}(t,x)dtdx=\Sigma((u(t,f)))dW(t)
\]
is the stochastic It\^{o} integral against the cylindrical Brownian motion.
\subsection{Assumptions}\label{sec-1}
We adopt assumptions from  \cite{FS17} or \cite{G98}.
The functions $b=b(t,x,r)$, $g=g(t,x,r)$, $\sigma=\sigma(t,x,r)$ are Borel functions of $(t,x,r)\in \mathbb{R}^+\times [0,1]\times \mathbb{R}$ satisfying the following conditions
\begin{description}
  \item[(H1)] There exists a constant $K>0$ such that for all $(t,x,r)\in[0,T]\times[0,1]\times \mathbb{R}$, we have
      \[
      \sup_{t\in[0,T]}\sup_{x\in[0,1]}|b(t,x,r)|\leq K(1+|r|).
      \]
  \item[(H2)] The function $g$ is of the form $g=g_1+g_2$, where $g_1$ and $g_2$ are Borel functions satisfying
  \[
  |g_1(t,x,r)|\leq K(1+|r|),\quad  |g_2(t,x,r)|\leq K(1+|r|^2).
  \]
  \item[(H3)] $\sigma$ is bounded and for every $T\geq0$, there exists a constant $L$ such that for all $(t,x,p,q)\in[0,T]\times[0,1]\times \mathbb{R}^2$, we have
      \[
      |\sigma(t,x,p)-\sigma(t,x,q)|\leq L|p-q|.
      \]
      Furthermore, $b$ and $g$ are locally Lipschitz with linearly growing Lipschitz constant, i.e.,
      \begin{eqnarray*}
      |b(t,x,p)-b(t,x,q)|\leq L(1+|p|+|q|)|p-q|,\\
      |g(t,x,p)-g(t,x,q)|\leq L(1+|p|+|q|)|p-q|.
      \end{eqnarray*}
\end{description}
\begin{dfn}\label{dfn-1}
 A random field $u$ is a solution to (\ref{e-1}) if $u=\{u(t,x), t\in \mathbb{R}^+, x\in [0,1]\}$ is an $L^2([0,1])-$valued continuous $\mathcal{F}_t-$adapted random field with initial value $f\in L^2([0,1])$ and satisfying for all $t\geq 0$, $\phi\in C^{2}([0,1])$ with $\phi(0)=0$, $\phi(1)=0$,
\begin{eqnarray}\notag
&&\int^1_0u(t,x)\phi(x)dx=\int^1_0f(x)\phi(x)dx+\int^t_0\int^1_0 u(s,x)\frac{\partial^2 \phi(x)}{\partial x^2}dxds
+\int^t_0\int^1_0b(s,x, u(s,x))\phi(x)dxds\\
\label{e-2}
&& \quad -\int^t_0\int^1_0g(s,x,u(s,x))\frac{\partial \phi(x)}{\partial x}dxds+\int^t_0\int^1_0\sigma(s,x,u(s,x))\phi(x)W(dxds), \ P-a.s.
\end{eqnarray}

\end{dfn}

The existence and uniqueness of the solution of (\ref{e-1}) is established in \cite{G98}. We recall it here.
\begin{thm}\label{thm-1}
Under assumptions (H1)-(H3), there exists a unique solution $u$ in the sense of Definition \ref{dfn-1}.
\end{thm}

\begin{remark}
Referring to Proposition 3.5 in \cite{G98}, under conditions in Theorem \ref{thm-1}, (\ref{e-2}) is equivalent to the following form. For all $t\geq 0$ and almost surely $\omega\in \Omega$,
 \begin{eqnarray}\notag
&&u(t,x)=\int^1_0 G_t(x,y)f(y)dy +\int^t_0\int^1_0 G_{t-s}(x,y)b(s,y,u(s,y))dyds\\
\label{e-3}
&&\quad \quad  -\int^t_0\int^1_0 \partial_y G_{t-s}(x,y)g(s,y,u(s,y))dyds+\int^t_0\int^1_0G_{t-s}(x,y)\sigma(s,y,u(s,y))W(dyds)
\end{eqnarray}
for almost every $x\in [0,1]$.
\end{remark}
\subsection{A lemma}
Define the linear operator $J$ by
\begin{eqnarray}\label{e-15}
J(v)(t,x)=\int^t_0\int^1_0H(r,t;x,y)v(r,y)dydr, \ t\in [0,T],\ x\in[0,1]
\end{eqnarray}
for every $v\in L^{\infty}([0,T];L^1([0,1]))$.

Referring to \cite{G98}, we have the following heat kernel estimate, which is very crucial to our proof.
\begin{lemma}
Let $J$ is defined by $H(s,t;x,y)=G_{t-s}(x,y)$ or by $H(s,t;x,y)=\frac{\partial G_{t-s}(x,y)}{\partial y}$ in (\ref{e-15}). Let $\rho\in[1,\infty]$, $q\in[1,\rho)$ and set $\kappa=1+\frac{1}{\rho}-\frac{1}{q}$. Then $J$ is a bounded linear operator from $L^{\gamma}([0,T];L^q([0,1]))$ into $C([0,T];L^{\rho}([0,1]))$ for $\gamma>2\kappa^{-1}$. Moreover, for any $T\geq 0$, there are constants $C_1$, $C_2$ such that
\begin{eqnarray}\label{e-16}
|J(v)(t,\cdot)|_{\rho}\leq C_1\int^t_0(t-s)^{\frac{\kappa}{2}-1}|v(s,\cdot)|_qds\leq C_2t^{\frac{\kappa}{2}-\frac{1}{\gamma}}\Big(\int^t_0|v(s,\cdot)|^{\gamma}_qds\Big)^{\frac{1}{\gamma}}.
\end{eqnarray}
\end{lemma}

\subsection{Statement of the main result}
In order to state the main result, we introduce some relevant notations and definitions.

Denote by $\mathcal{B}(H)$ the $\sigma-$field of all Borel subsets of $H$ and by $\mathcal{M}(H)$ the set of all probability measures defined on $(H,\mathcal{B}(H) )$.
Let $u(t,x,f)$ be the solution of (\ref{e-1}) and $P_t(f, \cdot)$ be the corresponding transition function
\[
P_t(f, \Gamma)=P(u(\cdot, t, f)\in \Gamma), \quad \Gamma \in \mathcal{B}(H), \ t>0,
\]
where $f$ is the initial condition.

For $\mu\in \mathcal{M}(H)$, we set
\[
P^*_t \mu(\Gamma)=\int_{H}P_t (x, \Gamma) \mu(dx)
\]
for $t\geq 0$ and $\Gamma \in \mathcal{B}(H)$.
\begin{dfn}
A probability measure $\mu\in \mathcal{M}(H)$ is said to be invariant or stationary with respect to $P_t$,  if and only if $P^*_t \mu=\mu$ for each $t\geq 0$.
\end{dfn}
Denote by ${B}_b(H)$ the space of all bounded measurable functions on $H$.
The semigroup $P_t$ associated with the solution $u(t,x,f)$ to (\ref{e-1}) is defined by
\[
P_t \psi(f)=E[\psi(u(t,f))],\ \psi\in {B}_b(H).
\]
\begin{dfn}
 $P_t$ is strong Feller, if $P_t$ maps $ B_b(H)$ into $ C_b(H)$ for $t>0$.
\end{dfn}
To obtain the strong Feller property of $P_t$, we need an additional condition:
\begin{description}
  \item[(H4)] There exists strictly positive constants $k_1$, $k_2$ such that $k_1\leq |\sigma(\cdot)|\leq k_2$.
\end{description}

\begin{thm}\label{thm-5}
Let $(P_t)_{t\geq 0}$ be the semigroup associated with the solution to (\ref{e-1}). Under assumptions (H1)-(H4), $(P_t)_{t\geq 0}$ is ergodic.
\end{thm}
\begin{proof}
Due to Theorem 3.2.6 in \cite{D-Z}, it suffices to prove the existence and uniqueness of invariant measures for $P_t$. We divide the proof into two parts. In the first part, we  prove the existence of invariant measure (see the following Sect. \ref{sec-2}). In the second part, we establish the uniqueness of invariant measures. According to Khas'minskii and Doob's theorem (see Theorem 4.1.1 and Theorem 4.2.1 in \cite{D-Z}), the uniqueness of invariant measures will be implied by strong Feller property and irreducibility. The proof process of them will be presented in the following Sect. \ref{sec-3}.
\end{proof}

\section{ Existence of Invariant Measures}\label{sec-2}
\begin{thm}\label{thm-2}
Suppose assumptions (H1)-(H3) are in force, then there exists an invariant measure to (\ref{e-1}) on $H$.
\end{thm}
\begin{proof}
According to the Krylov-Bogolyubov theorem (see \cite{D-Z}), if the family $\{P_t(f,\cdot); t\geq 1\}$ is tight, then
there exists an invariant measure for (\ref{e-1}). So we need to show that for any $\varepsilon>0$, there is a compact set $K\subset H$ such that
\[
P(u(t)\in K)\geq 1-\varepsilon \quad \forall t\geq 1,
\]
where $u(t)=u(t,f)=u(\cdot,t,f)$. For any $t\geq 1$, by the Markov property, we have
\[
P(u(t)\in K)=P\Big(u(1, u(t-1))\in K\Big).
\]
Hence, it is enough to show that $P\Big(u(1, u(t-1))\in K\Big)\geq 1-\varepsilon$ for all $t\geq 1$.

Define $\Phi_1: C([0,1];H)\rightarrow H$ by $u(1, \cdot, u(t-1))=\Phi_1(u(\cdot,\cdot, u(t-1)))$. Clearly, $\Phi_1$ is a continuous mapping. Thus, $K=\Phi_1(K')$ is a compact subset in $H$, if $K'$ is a compact subset in $C([0,1];H)$.
Recall Theorem 4.2 in \cite{FS17}, the tightness of $u(t,x,f)$ in $C([0,T];H)$ is obtained for any $f\in H$ and $T>0$. Due to $u(t-1)\in H$, then for any $\varepsilon>0$, there exists a compact subset $K'\subset C([0,1];H)$ such that
\[
P\Big(u(\cdot,\cdot, u(t-1))\in K'\Big)\geq 1-\varepsilon.
\]
Thus,
\begin{eqnarray*}
P\Big(u(1, \cdot, u(t-1))\in K\Big)&=&P\Big(u(1, \cdot, u(t-1))\in \Phi_1(K')\Big)\\
&\geq& P\Big(u(\cdot,\cdot, u(t-1))\in K'\Big)\\
&\geq & 1-\varepsilon,
\end{eqnarray*}
which implies the result.

\end{proof}

\section{Uniqueness of Invariant Measures}\label{sec-3}
\subsection{Strong Feller Property}
In this part, we aim to prove the following theorem.
\begin{thm}\label{thm-3}
Under assumptions (H1)-(H4), for any $t>0$, the semigroup $P_t$ is strong Feller.
\end{thm}

In \cite{G98}, Gy\"{o}ngy proves the existence and uniqueness of the solution to (\ref{e-1}) by an approximation procedure. Concretely, let $R>1$ be a positive number and consider
\begin{eqnarray}\notag
 \frac{\partial u^R(t,x)}{\partial t}&=&\frac{\partial^2 u^R(t,x)}{\partial x^2}+\kappa_R({|u^R(t)|^2_H})b(t,x,u^R(t,x))\\
 \label{e-5}
&&\ +\frac{\partial }{\partial x}\kappa_R({|u^R(t)|^2_H}) g(t,x, u^R(t,x))+\sigma(t,x,u^R(t,x))\dot{W}(t,x),\\
\label{e-6}
u^R(t,0)&=&u^R(t,1)=0, \\
\label{e-7}
u^R(0,\cdot)&=&f(\cdot)\in H,
\end{eqnarray}
where $\kappa_R=\kappa_R(r)$ is a $C^1(\mathbb{R})$ function such that $\kappa_R(r)=1$ for $|r|\leq R$, $\kappa_R(r)=0$ for $|r|> R+1$ and $\frac{d\kappa_R}{dr}\leq 2$ for all $r\in \mathbb{R}$.

By Proposition 4.7 in \cite{G98}, under conditions (H1)-(H3), for any $R>0$, there exists a unique solution $u^R\in C([0,T];H)$, $P-$a.s.. Define $\tau_R=\inf\{t\geq 0: |u^R(t)|^2_H\geq R\}$.
Notice that $u^S(t)=u^R(t)$ for $S\geq R$ and $t\leq \tau_R$. Therefore, we can set
\begin{eqnarray}\label{eqqq-3}
u(t)=u^R(t), \quad {\rm{if}}\ t\leq \tau_R.
\end{eqnarray}

Denote $P^{R}_t$ be the corresponding semigroup of $u^R(t,x)$, i.e., $P^{R}_t\psi(f)=E[\psi(u^{R}(t,f))]$, for any $\psi\in B_b(H)$. We claim that the following lemma holds.
\begin{lemma}
Under assumptions (H1)-(H3),
for $\psi\in B_b(H)$, we have
\begin{eqnarray}\label{e-11}
|P_t\psi(f)-P^R_t\psi(f)|\leq \frac{K'_T\|\psi\|_{\infty} (1+|f|^2_H)}{\log R},
\end{eqnarray}
where  $\|\psi\|_{\infty}=\sup_{f\in H}|\psi(f)|$ and $K'_T$ is a finite constant independent of $R$.
\end{lemma}
\begin{proof}
Define
\[
\eta^R(t,x)=\int^t_0\int^1_0G_{t-s}(x,y)\sigma(s,y, u^R(s,y))W(dyds),
\]
and
\[
\eta^{R,\ast}:=\sup_{(t,x)\in[0,T]\times[0,1]}|\eta^R(t,x)|.
\]
Using Theorem 2.1 in \cite{G98}, we deduce that, for any $p\geq 1$,
\[
\sup_{R> 1}E(\eta^{R,\ast})^p=\mu<\infty.
\]
Let $v^R=u^R-\eta^R$, which is a solution of the following equations
\begin{eqnarray}\notag
\frac{\partial v^R(t,x)}{\partial t}&=&\frac{\partial^2 v^R(t,x)}{\partial x^2}+\kappa_R({|u^R(t)|^2_H})b(t,x,v^R+\eta^R)\\
\label{e-8}
&&\quad +\frac{\partial }{\partial x}\kappa_R({|u^R(t)|^2_H}) g(t,x, v^R+\eta^R),\\
\label{e-9}
v^R(t,0)&=&v^R(t,1)=0, \\
\label{e-10}
v^R(0,\cdot)&=&f(\cdot).
\end{eqnarray}
Referring to Theorem 2.1 in \cite{G98}, there is a constant $K$  independent of $R$ such that
\begin{eqnarray}\label{eqqq-1}
|v^R(t)|^2_H\leq [|f|^2_H+Kt(1+|\eta^{R,\ast}|^4)]\exp\Big(K(1+|\eta^{R,\ast}|^2)t\Big)
\end{eqnarray}
holds for all $R>1$ and $t\in [0,T]$.

Hence, using (\ref{eqqq-1}), there exists a constant $C_1>0$ such that
\[
\sup_{t\in [0,T]}\log |u^R(t)|^2_H\leq \log\Big(|f|^2_H+KT(1+|\eta^{R,\ast}|^4)\Big)+K(1+|\eta^{R,\ast}|^2)T+|\eta^{R,\ast}|^2+C_1.
\]
By Jensen inequality, it follows that
\begin{eqnarray}\notag
E(\sup_{t\in [0,T]}\log |u^R(t)|^2_H)&\leq & \log\Big(|f|^2_H+KT(1+E|\eta^{R,\ast}|^4)\Big)+K(1+E|\eta^{R,\ast}|^2)T+E|\eta^{R,\ast}|^2+C_1\\
\label{e-30}
&\leq&K_T(\mu)(1+|f|^2_H),
\end{eqnarray}
where $K_T(\mu)$ is a finite number independent of $R$.

Since
\[
P(\tau_R\leq t)=\int_{\{\sup_{s\in[0,t]}\log|u^R(s)|^2_H\geq \log R\}}P(d\omega),
\]
by the Chebyshev inequality, we get
\[
P(\tau_R\leq t)\leq \frac{K_T(\mu)(1+|f|^2_H)}{\log R}.
\]
Let $\psi\in {B}_b(H)$, $f\in H$, we deduce from (\ref{eqqq-3}) that
\begin{eqnarray*}
|P_t\psi(f)-P^R_t\psi(f)|&=&|E\psi(u(t,f))-E\psi(u^R(t,f))|\\
&\leq& 2\|\psi\|_{\infty}P(\tau_R\leq t)\\
&\leq& \frac{K'_T\|\psi\|_{\infty} (1+|f|^2_H)}{\log R}.
\end{eqnarray*}
for a certain finite constant $K'_T$ independent of $R$.
\end{proof}
For any $R>1$, taking a non-negative function $\varphi\in C^{\infty}_0(\mathbb{R})$ with $\int_{\mathbb{R}}\varphi(x)dx=1$. Put
\begin{eqnarray*}
b_n(\xi)=n\int_{\mathbb{R}}\varphi(n(\xi-y))b(y)dy,\\
g_n(\xi)=n\int_{\mathbb{R}}\varphi(n(\xi-y))g(y)dy,\\
\sigma_{n}(\xi)=n\int_{\mathbb{R}}\varphi(n(\xi-y))\sigma(y)dy.
\end{eqnarray*}
Then, there exists $M>0$ such that
\[
\sup_{n}|b'_n(\xi)|\leq M(1+|\xi|), \quad \sup_{n,\xi}|\sigma'_n(\xi)|\leq M,\quad \sup_n|g'_n(\xi)|\leq M(1+|\xi|).
\]
Moreover, $b_n(\xi)\rightarrow b(\xi)$, $g_n(\xi)\rightarrow b(\xi)$ and $\sigma_n(\xi)\rightarrow \sigma(\xi)$, as $n\rightarrow \infty$. $u^{R,n}$ satisfies the following equations:
\begin{eqnarray*}
\frac{\partial u^{R,n}(t,x)}{\partial t}&=&\frac{\partial^2 u^{R,n}(t,x)}{\partial x^2}+\kappa_R({|u^{R,n}(t)|^2_H})b_n(t,x,u^{R,n}(t,x))\\
&&\ +\frac{\partial }{\partial x}\kappa_R({|u^{R,n}(t)|^2_H}) g_n(t,x, u^{R,n}(t,x))+\sigma_n(t,x,u^{R,n}(t,x))\dot{W}(t,x),\\
u^{R,n}(t,0)&=&u^{R,n}(t,1)=0, \\
u^{R,n}(0,\cdot)&=&f(\cdot).
\end{eqnarray*}
Referring to \cite{Z10}, one can verify that for any $f\in H$ and $R>0$,
\begin{eqnarray}\label{e-12}
\lim_{n\rightarrow \infty}\sup_{t\in[0,T]}E|u^{R,n}(t,f)-u^R(t,f)|_H=0.
\end{eqnarray}
For $\psi\in {B}_b(H)$, define $P^{R,n}_t\psi(f)=E[\psi(u^{R,n}(t,f))]$.
Hence, for any $f\in  H$,
\begin{eqnarray}\notag
|P^R_t\psi(f)-P^{R,n}_t\psi(f)|&=&|E\psi(u^{R}(t,f))-E\psi(u^{R,n}(t,f))|\\
\label{e-13}
&\leq& \|\psi\|_{\infty}E|u^{R,n}(t,f)-u^R(t,f)|_H.
\end{eqnarray}
According to (\ref{e-12}), we have that for any $R>1$,
\begin{eqnarray}\label{e-19}
\lim_{n\rightarrow \infty}|P^R_t\psi(f)-P^{R,n}_t\psi(f)|=0.
\end{eqnarray}
In the following, we aim to prove
\begin{lemma}\label{lem-4}
Suppose assumptions (H1)-(H4) are in force. Then for any $R>1$, $n>0$, there exists a constant $C(R,T,M,K)$ independent of $n$ such that for all $t\in (0,T]$, $\psi\in {B}_b(H)$ and $f_1, f_2\in H$,
\begin{eqnarray}\label{e-17}
|P^{R,n}_t\psi(f_1)-P^{R,n}_t\psi(f_2)|\leq \frac{C(R,T,M,K)}{\sqrt{t}}\|\psi\|_{\infty}|f_1-f_2|_H.
\end{eqnarray}
In particular, for every $R>1$, $n>0$, the semigroup $P^{R,n}_t$ is strong Feller on $H$.
\end{lemma}

\begin{proof} According to Lemma 7.1.5 in \cite{D-Z}, it suffices to prove for every $\psi\in C^2_b(H)$, the above equation (\ref{e-17}) holds.

Let $\mathcal{H}_{2,T}$ denote the Banach space of predictable $H-$valued processes $Y_t, t\geq 0$ with the norm:
\[
\|Y\|_{2,T}=\sup_{t\in[0,T]}\Big(E[|Y(t)|^2_H]\Big)^{\frac{1}{2}}.
\]
Since $b_n, g_n$ and $\sigma_n$ are smooth, $u^{R,n}(\cdot,f,\cdot)$ is continuously differential in $f$ as a mapping from $H$ to $\mathcal{H}_{2,T}$. Moreover, denote by $Y^{R,n}(t,f,h,x)=[Du^{R,n}(\cdot,f,\cdot)(h)](t,x)$ the directional derivative of $u^{R,n}$ at $f$ in the direction $h$. Then it satisfies that
\begin{eqnarray}\notag
Y^{R,n}(t,f,h,x)&=&\int^1_0G_t(x,y)h(y)dy+\int^t_0\int^1_0G_{t-s}(x,y)\kappa_R(|u^{R,n}|^2_H)b'_n(u^{R,n}(s,f,y))Y^{R,n}(s,f,g,y)dyds\\
\notag
&&\ +2\int^t_0\int^1_0G_{t-s}(x,y)\kappa'_R(|u^{R,n}|^2_H)(u^{R,n},Y^{R,n})_Hb_n(u^{R,n}(s,f,y))dyds\\ \notag
&&\ -\int^t_0\int^1_0\partial_yG_{t-s}(x,y)\kappa_R(|u^{R,n}|^2_H)g'_n(u^{R,n}(s,f,y))Y^{R,n}(s,f,g,y)dyds\\ \notag
&&\ -2\int^t_0\int^1_0\partial_yG_{t-s}(x,y)\kappa'_R(|u^{R,n}|^2_H)(u^{R,n},Y^{R,n})_{H} g_n(u^{R,n}(s,f,y))dyds\\
\label{e-14}
&&\ +\int^t_0\int^1_0G_{t-s}(x,y)\sigma'_n(u^{R,n}(s,f,y))Y^{R,n}(s,f,g,y)W(dyds).
\end{eqnarray}
In view of (\ref{e-14}), we have for $t\leq T$,
\begin{eqnarray*}
E[|Y^{R,n}(t,f,h,\cdot)|^2_H]
&=&E[\int^1_0|Y^{R,n}(t,f,h,x)|^2dx]\\
&\leq& E\int^1_0\Big(\int^1_0G_t(x,y)h(y)dy\Big)^2dx\\
&&\ +E\int^1_0\Big(\int^t_0\int^1_0G_{t-s}(x,y)\kappa_R(|u^{R,n}|^2_H)b'_n(u^{R,n}(s,f,y))Y^{R,n}(s,f,g,y)dyds\Big)^2dx\\
&&\ +4 E\int^1_0(\int^t_0\int^1_0G_{t-s}(x,y)\kappa'_R(|u^{R,n}|^2_H)(u^{R,n},Y^{R,n})_Hb_n(u^{R,n}(s,f,y))dyds)^2dx\\
&&\ +E\int^1_0\Big(\int^t_0\int^1_0\partial_yG_{t-s}(x,y)\kappa_R(|u^{R,n}|^2_H)g'_n(u^{R,n}(s,f,y))Y^{R,n}(s,f,g,y)dyds\Big)^2dx\\
&&\ +4E\int^1_0\Big(\int^t_0\int^1_0\partial_yG_{t-s}(x,y)\kappa'_R(|u^{R,n}|^2_H)(u^{R,n},Y^{R,n})_Hg_n(u^{R,n}(s,f,y))dyds\Big)^2 dx\\
&&\ +E\int^1_0\Big(\int^t_0\int^1_0G_{t-s}(x,y)\sigma'_n(u^{R,n}(s,f,y))Y^{R,n}(s,f,g,y)W(dyds)\Big)^2dx\\
&:=& I_1+I_2+I_3+I_4+I_5+I_6.
\end{eqnarray*}
Using the heat kernal estimates, we get
\[
I_1\leq C_T |h|^2_H.
\]
When $|u^{R,n}|^2_H\leq R$, applying (\ref{e-16}), it follows that
\begin{eqnarray*}
I_2&\leq&  E\int^1_0dx\Big[\int^t_0(t-s)^{-\frac{3}{4}}|b'_n(u^{R,n})Y^{R,n}|_1ds\Big]^2\\
&\leq&(1+R)M^2 E\Big[\int^t_0(t-s)^{-\frac{3}{4}}|Y^{R,n}|_Hds\Big]^2\\
&\leq&(1+R)M^2 E\Big(\int^t_0(t-s)^{-\frac{3}{4}}ds\Big)\Big[\int^t_0(t-s)^{-\frac{3}{4}}|Y^{R,n}|^2_Hds\Big]\\
&\leq&4M^2t^{\frac{1}{4}}(1+R)\int^t_0(t-s)^{-\frac{3}{4}}E|Y^{R,n}|^2_Hds.
\end{eqnarray*}
When $|u^{R,n}|^2_H\leq R$, with the help of H\"{o}lder inequality and $\frac{d\kappa_R(r)}{dr}\leq 2$, we deduce that
\begin{eqnarray*}
I_3&\leq& 4 E\int^1_0dx\Big(\int^t_0\int^1_0G^2_{t-s}(x,y)dyds\Big)\Big(\int^t_0\int^1_0(u^{R,n},Y^{R,n})^2_H|b_n(u^{R,n})|^2dyds\Big)\\
&\leq& 4K^2 E\int^{t}_0\frac{1}{\sqrt{t-s}}ds\int^t_0(u^{R,n},Y^{R,n})^2_H(1+|u^{R,n}|^2_H)ds\\
&\leq& 2K^2t^{\frac{1}{2}}E\int^t_0(u^{R,n},Y^{R,n})^2_H(1+|u^{R,n}|^2_H)ds\\
&\leq& 2K^2t^{\frac{1}{2}}E\int^t_0|u^{R,n}|^2_H|Y^{R,n}|^2_H(1+|u^{R,n}|^2_H)ds\\
&\leq&K^2R(1+R)t^{\frac{1}{2}}\int^t_0E|Y^{R,n}|^2_Hds.
\end{eqnarray*}
When $|u^{R,n}|^2_H\leq R$, applying (\ref{e-16}), it follows that
\begin{eqnarray*}
I_4&\leq&  E\int^1_0dx\Big[\int^t_0(t-s)^{-\frac{3}{4}}|g'_n(u^{R,n})Y^{R,n}|_1ds\Big]^2\\
&\leq&(1+R)M^2E\Big[\int^t_0(t-s)^{-\frac{3}{4}}|Y^{R,n}|_Hds\Big]^2\\
&\leq&(1+R)M^2E\Big(\int^t_0(t-s)^{-\frac{3}{4}}ds\Big)\Big[\int^t_0(t-s)^{-\frac{3}{4}}|Y^{R,n}|^2_Hds\Big]\\
&\leq&4M^2t^{\frac{1}{4}}(1+R)\int^t_0(t-s)^{-\frac{3}{4}}E|Y^{R,n}|^2_Hds.
\end{eqnarray*}
Similar to the proof of $I_3$, we get
\begin{eqnarray*}
I_5&\leq& 4 E\int^1_0dx\Big[\int^t_0(t-s)^{-\frac{3}{4}}|(u^{R,n},Y^{R,n})_Hg_n(u^{R,n})|_1ds\Big]^2\\
&\leq& 4K^2R(1+R)^2E\Big[\int^t_0(t-s)^{-\frac{3}{4}}|Y^{R,n}|_Hds\Big]^2\\
&\leq&16t^{\frac{1}{4}}K^2R(1+R)^2\int^t_0(t-s)^{-\frac{3}{4}}E|Y^{R,n}|^2_Hds.
\end{eqnarray*}
By It\^{o} isometry, it follows that
\begin{eqnarray*}
I_6&\leq& E\int^1_0\int^t_0\int^1_0G^2_{t-s}(x,y)|Y^{R,n}|^2dydsdx\\
&\leq& E\int^t_0\int^1_0\Big(\int^1_0G^2_{t-s}(x,y)dx\Big)|Y^{R,n}|^2dyds\\
&\leq& E\int^t_0\frac{1}{\sqrt{t-s}}\int^1_0|Y^{R,n}|^2dyds\\
&\leq& \int^t_0\frac{1}{\sqrt{t-s}}E|Y^{R,n}|^2_Hds.
\end{eqnarray*}
Based on the above estimates, we deduce that
\begin{eqnarray*}
E[|Y^{R,n}(t,\cdot)|^2_H]&\leq& C_T|h|^2_H+K^2R(1+R)t^{\frac{1}{2}}\int^t_0E|Y^{R,n}|^2_Hds\\
&&\ +[8M^2t^{\frac{1}{4}}(1+R^2)+16t^{\frac{1}{4}}K^2R^2(1+R)^2]\int^t_0(t-s)^{-\frac{3}{4}}E|Y^{R,n}|^2_Hds\\
&&\ +\int^t_0\frac{1}{\sqrt{t-s}}E|Y^{R,n}|^2_Hds.
\end{eqnarray*}
Applying Gronwall inequality, we get
\begin{eqnarray*}
E[|Y^{R,n}(t,\cdot)|^2_H]\leq C_T|h|^2_H\cdot \exp\Big\{K^2R(1+R)t^{\frac{3}{2}}+32M^2t^{\frac{1}{2}}(1+R)+64t^{\frac{1}{2}}K^2R(1+R)^2+2t^\frac{1}{2}\Big\}.
\end{eqnarray*}
Hence, it gives that
\begin{eqnarray*}
\sup_{t\in[0,T]}E[|Y^{R,n}(t,\cdot)|^2_H]\leq \tilde{C}(R,T,M,K)|h|^2_H,
\end{eqnarray*}
where
\[
\tilde{C}(R,T,M,K):=C_T\exp\Big\{2T^{\frac{3}{2}}+K^2R(1+R)T^{\frac{3}{2}}+32M^2T^{\frac{1}{2}}(1+R)+64T^{\frac{1}{2}}K^2R(1+R)^2+2T^\frac{1}{2}\Big\}
\]

Let $\psi\in C^2_b(H)$. By the Elworthy formula (see Lemma 7.1.3 in \cite{D-Z}), we have
\begin{eqnarray*}
&&\langle DP^{R,n}\psi(f),h\rangle\\
&=&\frac{1}{t}E\{\psi(u^{R,n}(t,f))\int^t_0\langle (\Sigma_n)^{-1}(u^{R,n}(s,f))[Du^{R,n}(\cdot,f,\cdot)(h)](s,\cdot),W(ds)\rangle\}.
\end{eqnarray*}
It follows from assumption (H4) that
\begin{eqnarray*}
&&|\langle DP^{R,n}\psi(f),h\rangle|^2\\
&\leq&\frac{1}{t^2}\|\psi\|^2_{\infty}E\Big\{\int^t_0 |(\Sigma_n)^{-1}(u^{R,n}(s,f))|^2|[Du^{R,n}(\cdot,f,\cdot)(h)]|^2_Hds\Big\}\\
&\leq&\frac{1}{k^2_1t^2}\|\psi\|^2_{\infty}E\Big\{\int^t_0 |Y^{R,n}(t,f,h,\cdot)|^2_Hds\Big\}\\
&\leq&\frac{\tilde{C}(R,T,M,K)}{k^2_1t}\|\psi\|^2_{\infty}|h|^2_H.
\end{eqnarray*}
which implies (\ref{e-17}).
\end{proof}

Recall Theorem 5.4 in \cite{F-M}, which is a general criterion proposed by Flandoli and Romito to establish $\mathcal{W}-$strong Feller property of the semigroup associated with SPDEs.  It says that if a Markov process coincides on a small positive random time with a strong Feller process, then it is strong Feller itself.  Here, we use the version of $\mathcal{W}=H$.

\begin{thm}\label{thm-8}
\textbf{(Weak-strong uniqueness)} Let $(P_f)_{f\in H}$ be an a.s. Markov process on $(\Omega, \mathcal{B})$ and for each $R>0$, let $P^R_f$ be an a.s. $H-$Markov process on $(\Omega, \mathcal{B})$.
 For every $f\in H$, the system (\ref{e-5})-(\ref{e-7}) has a unique solution $P^R_f$, with
\[
P^R_f[C([0,\infty);H]=1.
\]
Let $\tau_R:\Omega\rightarrow [0,\infty]$ be defined by
\[
\tau_R(\omega):= \inf\{t\geq 0: |\omega(t)|^2_{H}\geq R\}
\]
and $\tau_R(\omega):= \infty$ if this set is empty.
If $f\in H$ and $|f|^2_{H}< R$, then
\begin{eqnarray}\label{eq-39}
\lim_{\varepsilon \rightarrow 0} P^R_{f+h}[\tau_R\geq \varepsilon]=1, \ {\rm{uniformly \ in}} \ h\in H, \ |h|_{\mathcal{W}}<1.
\end{eqnarray}
Moreover,
\begin{eqnarray}\label{eq-40}
\mathbb{E}^{P^R_f}[\varphi(\omega_t)I_{[\tau_R \geq t]}]=\mathbb{E}^{P_f}[\varphi(\omega_t)I_{[\tau_R \geq t]}]
\end{eqnarray}
 for every $t\geq 0$ and $\varphi \in B_b(H)$.

 If for every $R>0$, the transition semigroup $(P^R_t)_{t\geq 0}$  is $H$-strong Feller, then $(P_t)_{t\geq 0}$ is $H$-strong Feller.
\end{thm}

Now, we are able to prove Theorem \ref{thm-3}.

\noindent\textbf{Proof of Theorem \ref{thm-3}}\quad
Thanks to (\ref{e-17}), for any $R>1$, $n>0$ and $\psi\in B_b(H)$, $P^{R,n}_t\psi$ are continuous functions on $H$. Moreover, using (\ref{e-19}), we get for any $R>1$, $P^{R,n}_t\psi\rightarrow P^R_t\psi$ uniformly on bounded sets, as $n\rightarrow \infty$. Hence, for any $R>1$, $P^R_t\psi$ is continuous, i.e., for any $R>1$, $P^R_t$ is strong Feller on $H$. To obtain $H-$strong Feller of $P_t$, due to Theorem \ref{thm-8}, we need to verify (\ref{eq-39}) and (\ref{eq-40}).

In order to prove (\ref{eq-39}), it is sufficient to show that $P^R_f[\tau_R<\varepsilon]\leq C(\varepsilon, R)$ with $C(\varepsilon, R)\downarrow 0$ as $\varepsilon\downarrow 0$, for all $f\in H$ with $|f|^2_H\leq \frac{R}{8}$.

Define
\[
\eta^R(t,x)=\int^t_0\int^1_0G_{t-s}(x,y)\sigma(u^R(s,y))W(dyds).
\]
Referring to Theorem 2.1 in \cite{G98}, it gives that, for every $p\geq 1$,
\begin{eqnarray}\label{eqq-1}
\sup_{R> 1}E\Big(\sup_{(t,x)\in[0,T]\times [0,1]}|\eta^R(t,x)|^p\Big)<\infty.
\end{eqnarray}
Let $v^R=u^R-\eta^R$, it satisfies (\ref{e-8})-(\ref{e-10})
and for all $R>1$, $t\in [0,T]$,
\begin{eqnarray}\label{eqq-2}
|v^R(t)|^2_H\leq [|f|^2_H+Kt(1+\sup_{s\in [0,t]}\sup_{x\in [0,1]}|\eta^R(s,x)|^4)]\exp\Big\{K(1+\sup_{s\in [0,t]}\sup_{x\in [0,1]}|\eta^R(s,x)|^2)t\Big\}.
\end{eqnarray}

Now, for a certain small $\varepsilon$ will be determined later, let
\[
\Theta_{\varepsilon, R}:= \sup_{t\in[0,\varepsilon]}\sup_{x\in [0,1]}|\eta^R(t,x)|
\]
and assume that $\Theta^{2}_{\varepsilon, R}\leq \frac{R}{8}$.

Based on (\ref{eqq-2}), we deduce that
\begin{eqnarray*}
|u^R(t)|^2_H&\leq& 2|v^R(t)|^2_H+2\sup_{s\in [0,t]}\sup_{x\in [0,1]}|\eta^R(s,x)|^2\\
&\leq& 2[|f|^2_H+Kt(1+\sup_{s\in [0,t]}\sup_{x\in [0,1]}|\eta^R(s,x)|^4)]\exp\Big\{K(1+\sup_{s\in [0,t]}\sup_{x\in [0,1]}|\eta^R(s,x)|^2)t\Big\}\\
&&\ +2\sup_{s\in [0,t]}\sup_{x\in [0,1]}|\eta^R(s,x)|^2.
\end{eqnarray*}
Then, it follows that
\begin{eqnarray*}
\sup_{t\in [0,\varepsilon]}|u^R(t)|^2_H
&\leq& 2[\frac{R}{8}+K\varepsilon(1+\frac{R^2}{64})]\exp\Big\{K(1+\frac{R}{8})\varepsilon\Big\}+\frac{R}{4},
\end{eqnarray*}
hence, we can choose small enough $\varepsilon$  such that $\sup_{t\in [0,\varepsilon]}|u^R(t)|^2_H\leq R$. Thus,
\begin{eqnarray*}
P^R_f(\tau_R<\varepsilon)\leq P^R_f(\sup_{t\in[0,\varepsilon]}\sup_{x\in [0,1]}|\eta^R(t,x)|^2>\frac{R}{8}).
\end{eqnarray*}
Letting $\varepsilon\downarrow 0$, taking into account of (\ref{eqq-1}), we have $P^R_f(\tau_R<\varepsilon)\rightarrow 0$.

It remains to establish (\ref{eq-40}). From the proof process of Theorem 2.1 in \cite{G98}, we get $u(t)=u^R(t)$ on the time interval $[0, \tau_R(u)\wedge\tau_R(u^R)]$, $P-$a.s., for every $t\geq 0$. Moreover, the solution $u$ is $H-$valued weakly continuous in time, we obtain $\tau_R(u^R)=\tau_R(u)$. Hence, $u(t\wedge \tau_R(u))=u^R(t\wedge \tau_R(u^R))$. Based on the above, we complete the proof.



\subsection{Irreducibility}
For given $a\in H$ and $r>0$, let $B_H(a,r)$ stand for the ball $\{z\in H:|a-z|_H<r\}$. Note that $u$ is irreducible if and only if for all $x\in H$, $t>0$, $a\in  H$ and $r>0$,
\begin{eqnarray}\label{eqqq-2}
P(u(t,x)\in B_H(a,r))>0.
\end{eqnarray}
From now on, $x,t,a$ and $r$ are fixed.

The main result in this part is
\begin{thm}\label{thm-4}
 Assume assumptions (H1)-(H4) hold. The semigroup $P_t$ is irreducible for any $t>0$.
\end{thm}

According to Theorem 2.1 in \cite{G98}, the solution $u^R(t,x)$ of (\ref{e-5})-(\ref{e-7}) converges to $u(t,x)$ in $C([0,T]; H)$ in probability as $R\rightarrow \infty$. Hence, there exits a large $R>0$ such that
\begin{eqnarray}\label{e-27}
P\Big(|u(t,x)-u^R(t,x)|_H\geq\frac{r}{2}\Big)\leq \frac{1}{4}.
\end{eqnarray}
Fix $R$ determined by (\ref{e-27}). In the following, we aim to study the solution $u^R(t,x)$ of (\ref{e-5})-(\ref{e-7}).

Similar to the proof of Lemma 3.1 in \cite{P95}, using Girsanov theorem, we can obtain
\begin{lemma}\label{lem-1}
Assume the conditions of Theorem \ref{thm-4} are in force. Let $t_1\in (0,t)$ and let $f: [t_1,t]\times H\rightarrow H$ be a bounded measurable mapping. If $Z^R$ is the solution of the following equation
\begin{eqnarray}\label{e-22}
\left\{
  \begin{array}{ll}
   \partial_s Z^R(s)=\frac{\partial^2 Z^R(s) }{\partial x^2}+\partial_x \kappa_R(|Z^R(s)|^2_H)g(Z^R(s))+\sigma(Z^R(s))\dot{W}(s,x), & {\rm{on}}\ [0,t_1], \\
   \partial_s Z^R(s)=\frac{\partial^2 Z^R(s) }{\partial x^2}+f(s,Z^R(t_1))+\partial_x \kappa_R(|Z^R(s)|^2_H) g(Z^R(s))+\sigma(Z^R(s))\dot{W}(s,x) , & {\rm{on}} \ (t_1, t],
  \end{array}
\right.
\end{eqnarray}
with $Z^R(0)=\gamma$, then the laws in $H$ of $u^R(t,\gamma)$ and $Z^R(t)$ are equivalent.
\end{lemma}

\begin{prp}\label{prp-1}
Assume assumptions (H1)-(H4) hold. For the solution $u^R(t,x)$ of (\ref{e-5})-(\ref{e-7}), we have
\[
P\Big(|u^R(t,x)-a|_H<\frac{r}{2}\Big)\geq \frac{1}{2},
\]
where $a,r, t,x$ is fixed by (\ref{eqqq-2}) and $R$ is determined by (\ref{e-27}).
\end{prp}
\begin{proof}
According to Lemma \ref{lem-1}, we need to show that there exists a function $f$ satisfying the assumptions specified in Lemma \ref{lem-1} such that for the corresponding solution $Z^R$ satisfying $P(|Z^R(t,x)-a|_H<\frac{r}{2})\geq \frac{1}{2}$.

Denote by $\tilde{Z}^R$ the solution of the equation
\begin{eqnarray*}
\partial_s \tilde{Z}^R(s)=\frac{\partial^2 \tilde{Z}^R(s)}{\partial x^2}+\partial_x \kappa_R(|\tilde{Z}^R(s)|^2_H)g(\tilde{Z}^R(s))+\sigma(\tilde{Z}^R(s))\dot{W}(s,x),
\end{eqnarray*}
with $\tilde{Z}^R(0)=\gamma\in H$. Then, we have
\begin{eqnarray*}
\tilde{Z}^R(s,x)&=&\int^1_0G_s(x,y)\gamma(y)dy-\int^s_0\int^1_0\partial_y G_{s-r}(x,y)\kappa_R(|\tilde{Z}^R(r,y)|^2_H)g(\tilde{Z}^R(r,y))dydr\\
&&\ +\int^s_0\int^1_0 G_{s-r}(x,y)\sigma(Z^R(r,y))W(dydr).
\end{eqnarray*}
Using It\^{o} isometry, we get
\begin{eqnarray*}
E|\tilde{Z}^R(s)|^2_H&=&E\int^1_0|\tilde{Z}^R(s)|^2dx\\
&\leq& C|\gamma|^2_H+E\int^1_0[\int^s_0(s-r)^{-\frac{3}{4}}|\kappa_R(|\tilde{Z}^R(r)|^2_H)g(\tilde{Z}^R(r))|_1dr]^2dx\\
&&\ +E\int^1_0\int^s_0\int^1_0G^2_{s-r}(x,y)|\sigma(Z^R(r,y))|^2dydr dx\\
&\leq& C|\gamma|^2_H+K^2(16+R^2)^2s^{\frac{1}{2}}+2k^2_2s^\frac{1}{2}.
\end{eqnarray*}
Then, it gives that
\begin{eqnarray*}
\sup_{0\leq s\leq t}E|\tilde{Z}^R(s)|^2_H
\leq C|\gamma|^2_H+K^2(16+R^2)^2t^{\frac{1}{2}}+2k^2_2t^\frac{1}{2}.
\end{eqnarray*}
Thus, there exists a constant $K>0$ such that
\begin{eqnarray}\label{e-21}
\sup_{0\leq s\leq t}E|\tilde{Z}^R(s)|^2_H
\leq \frac{K^2}{4}.
\end{eqnarray}
Taking an element $\tilde{a}\in D(A)$ such that $|a-\tilde{a}|_H<\frac{r}{6}$. From now on, $K$ and $\tilde{a}$ are fixed. For any $\iota<t$, let us denote by $f_{\iota}$ a bounded measurable extension of the function $\tilde{f}_{\iota}$ defined by
\begin{eqnarray*}
\tilde{f}_{\iota}(s,\xi)=\left\{
                          \begin{array}{ll}
                            0, & {\rm{if}}\ |\xi|_H\geq 2K, \\
                            \frac{1}{t-\iota}\int^1_0G_{s-\iota}(y,z)(\tilde{a}-\xi(z))dz-A\tilde{a}, & {\rm{if}} \ |\xi|_H\leq K.
                          \end{array}
                        \right.
\end{eqnarray*}
Obviously, we can assume that there exist constants $C_1$, $C_2$ and extensions $f_{\iota}$ of $\tilde{f}_{\iota}$ such that
\begin{eqnarray}\label{e-20}
|f_{\iota}(s,y)|\leq C_1(t-\iota)^{-1}+C_2\quad {\rm{for\ all}}\ 0\leq \iota<t,\ s\in [\iota,t], \xi\in H.
\end{eqnarray}
Hence, for a certain $t_1<t$, the function $f=f_{t_1}$ has the desired properties.

Let $Z^R_{\iota}(s)$ be the solution of (\ref{e-22}) on $(0,t)$ with $f=f_{\iota}$ and $t_1=\iota$.
Due to (\ref{e-21}) and (\ref{e-20}), we have
\begin{eqnarray}\label{e-23}
\sup_{0\leq \iota<t}\sup_{0\leq s\leq t}E|Z^R_{\iota}(s)|^2_H<\infty.
\end{eqnarray}
Utilizing the heat kernal estimates, it follows that
\begin{eqnarray*}
E|\int^t_{t_1}\int^1_0G_{t-s}(x,y)\sigma(Z^R_{t_1}(s))W(dyds)|^2_H\leq CK^2(t-t_1)^{\frac{1}{2}}\sup_{0\leq \iota<t}\sup_{0\leq s\leq t}E|Z^R_{\iota}(s)|^2_H,
\end{eqnarray*}
and
\begin{eqnarray*}
E|\int^t_{t_1}\int^1_0\partial_y G_{t-s}(x,y)g(Z^R_{t_1}(s))dyds|^2_H\leq CK(t-t_1)^{\frac{1}{4}}\sup_{0\leq \iota<t}\sup_{0\leq s\leq t}E|Z^R_{\iota}(s)|^2_H.
\end{eqnarray*}
Taking into account the equation (\ref{e-23}), we can find $t_1<t$ such that
\begin{eqnarray*}
E|\int^t_{t_1}\int^1_0G_{t-s}(x,y)\sigma(Z^R_{t_1}(s))W(dyds)|^2_H\leq \frac{r^2}{288},
\end{eqnarray*}
and
\begin{eqnarray*}
E|\int^t_{t_1}\int^1_0\partial_y G_{t-s}(x,y)g(Z^R_{t_1}(s))dyds|^2_H\leq\frac{r^2}{288}.
\end{eqnarray*}
With the aid of Chebyshev inequality, we deduce that
\begin{eqnarray}\label{e-24}
P\Big(|\int^t_{t_1}\int^1_0G_{t-s}(x,y)\sigma(Z^R_{t_1}(s))W(dyds)|^2_H\geq \frac{r}{6}\Big)&\leq& \frac{1}{8},\\
\label{e-25}
P\Big(|\int^t_{t_1}\int^1_0\partial_y G_{t-s}(x,y)g(Z^R_{t_1}(s))dyds|^2_H\geq \frac{r}{6}\Big)&\leq&\frac{1}{8}.
\end{eqnarray}
Note that for all $\xi\in H$ such that $|\xi|_H\leq K$, we have
\begin{eqnarray}\notag
&&\int^1_0G_{t-t_1}(x,y)\xi(y)dy+\int^t_{t_1}\int^1_0G_{t-s_1}(x,y)f_{t_1}(s_1,\xi(y))dyds_1\\ \notag
&=&\int^1_0G_{t-t_1}(x,y)\xi(y)dy+\frac{1}{t-t_1}\int^t_{t_1}\int^1_0G_{t-s_1}(x,y)\int^1_0G_{s_1-t_1}(y,z)(\tilde{a}-\xi(z))dzdyds_1\\ \notag
&&\ -\int^t_{t_1}\int^1_0AG_{t-s_1}(x,y)\tilde{a}dyds_1\\
\label{e-26}
&=& \int^1_0G_{t-t_1}(x,y)\tilde{a}dy+\int^t_{t_1}\frac{d}{ds_1}[\int^1_0  G_{t-s_1}(x,z)dz \tilde{a}] ds_1=\tilde{a},
\end{eqnarray}
where (\ref{e-29}) and (\ref{e-28}) are used.

Set $Z^R=Z^R_{t_1}$ and $f=f_{t_1}$, then $Z^R$ satisfies
\begin{eqnarray*}
Z^R(t,x)&=&\int^1_0G_{t-t_1}(x,y)Z^R(t_1,y)dy+\int^t_{t_1}\int^1_0G_{t-s}(x,y)f(s,Z^R(t_1,y))dyds\\
&&\ -\int^t_{t_1}\int^1_0\partial_y G_{t-s}(x,y)g(Z^R(s,y))dyds+\int^t_{t_1}\int^1_0G_{t-s}(x,y)\sigma(Z^R(s,y))W(dyds)\\
&:=& I_1+I_2+I_3.
\end{eqnarray*}
Applying (\ref{e-26}) with $\xi(\cdot)=Z^R(t_1,\cdot)$ and (\ref{e-21}), we deduce that
\[
P(I_1=\tilde{a})\geq P(|Z^R(t_1)|_H\leq K)\geq \frac{3}{4}.
\]
From (\ref{e-24}) and (\ref{e-25}), we get
\[
P(|I_2|_H\geq\frac{r}{6})\leq \frac{1}{8}, \quad P(|I_3|_H\geq \frac{r}{6})\leq \frac{1}{8}.
\]
Consequently, as $|\tilde{a}-a|_H<\frac{r}{6}$, we get
\begin{eqnarray*}
P\Big(Z^R(t)\in B_H(a,\frac{r}{2})\Big)&=&P\Big(|Z^R(t)-a|_H<\frac{r}{2}\Big)\\
&=&P\Big(|(I_1-\tilde{a})-I_2+I_3+\tilde{a}-a|_H<\frac{r}{2}\Big)\\
&\geq& P\Big(I_1=\tilde{a}, |I_2|_H< \frac{r}{6}, |I_3|_H< \frac{r}{6}\Big)\\
&\geq& P\Big(I_1=\tilde{a})-P(|I_2|_H\geq \frac{r}{6}\Big)-P\Big(|I_3|_H\geq \frac{r}{6}\Big)\\
&\geq& \frac{3}{4}-\frac{1}{8}-\frac{1}{8}=\frac{1}{2}.
\end{eqnarray*}
We complete the proof.
\end{proof}

Now, we are able to prove Theorem \ref{thm-4}.

\noindent\textbf{Proof of Theorem \ref{thm-4}}.\quad Taking into account  (\ref{e-27}) and Proposition \ref{prp-1}, we get
\begin{eqnarray*}
P(u(t)\in B_H(a,r))&=&P(|u(t)-a|_H<r)\\
&\geq& P\Big(|u(t)-u^R(t)|_H<\frac{r}{2}, \ |u^R(t)-a|_H<\frac{r}{2}\Big)\\
&\geq& P\Big( |u^R(t)-a|_H<\frac{r}{2}\Big)-P\Big(|u(t)-u^R(t)|_H\geq\frac{r}{2}\Big)\\
&=& P\Big( |Z^R(t)-a|_H<\frac{r}{2}\Big)-P\Big(|u(t)-u^R(t)|_H\geq\frac{r}{2}\Big)\\
&\geq& \frac{1}{2}-\frac{1}{4}=\frac{1}{4}>0,
\end{eqnarray*}
which implies the result.

\section{Application to examples}
The main results can be applied to the following stochastic nonlinear evolution equations:
 \begin{description}
   \item[(1)] If $f=0, g(t,x,r)=\frac{1}{2}r^2, \sigma\neq 0$, then (\ref{e-1}) is a stochastic Burgers equation.
 \end{description}
It arose in the connection with the study of turbulent fluid motion and its ergodicity has been established by Da Prato and Gatarek  in \cite{D-G}.
\begin{description}
  \item[(2)] If $f\neq 0, g=0, \sigma\neq 0$, then (\ref{e-1}) is a stochastic reaction-diffusion equation.
\end{description}
This model has been studied by Cerrai \cite{C02}, Funaki \cite{F83} and so on. In particular, Cerrai \cite{C02} proved the existence of invariant measures.

\

\noindent {\bf  Acknowledgements}\   This work was partly supported by National Natural Science Foundation of China (NSFC) (No. 11431014, 11801032),
Key Laboratory of Random Complex Structures and Data Science, Academy of Mathematics and Systems Science, Chinese Academy of Sciences(No. 2008DP173182), China Postdoctoral Science Foundation funded project (No. 2018M641204).

\def\refname{ References}

\end{document}